\documentclass[12pt]{amsart}
\usepackage[utf8]{inputenc}
\usepackage{hyperref}
\hypersetup{colorlinks=true,allcolors=blue}
\def\Snospace~{\S{}}

\usepackage{aliascnt}
\usepackage{setspace}
\usepackage{amsmath}
\usepackage{amsthm}
\usepackage{amssymb}
\usepackage{mathtools}
\usepackage{amsfonts}
\usepackage{enumitem}
\usepackage{wasysym}
\usepackage{graphicx}
\usepackage{mathrsfs}
\usepackage[nobysame]{amsrefs}
\usepackage[textsize=scriptsize]{todonotes}
\usepackage{microtype}
\usepackage{textcomp}
\usepackage{lipsum}
\usepackage{csquotes}
\usepackage{lmodern}
\usepackage[T1]{fontenc}
\usepackage[all]{nowidow}
\usepackage{booktabs}

\newtheorem{theorem}{Theorem}
\numberwithin{theorem}{section}
\newaliascnt{lemma}{theorem}
\newtheorem{lemma}[lemma]{Lemma}
\aliascntresetthe{lemma}

\newaliascnt{proposition}{theorem}

\aliascntresetthe{proposition}

\newaliascnt{corollary}{theorem}
\newtheorem{corollary}[corollary]{Corollary}
\aliascntresetthe{corollary}

\newaliascnt{conjecture}{theorem}

\aliascntresetthe{conjecture}

\newaliascnt{definition}{theorem}

\aliascntresetthe{definition}

\newaliascnt{example}{theorem}

\aliascntresetthe{example}

\newaliascnt{remark}{theorem}
\theoremstyle{remark}
\newtheorem{remark}[remark]{Remark}
\aliascntresetthe{remark}

\numberwithin{figure}{section}

\newcommand{\NN}{\mathbb{N}}
\newcommand{\ZZ}{\mathbb{Z}}
\newcommand{\QQ}{\mathbb{Q}}
\newcommand{\RR}{\mathbb{R}}
\newcommand{\CC}{\mathbb{C}}
\newcommand{\HH}{\mathbb{H}}

\DeclareMathOperator{\GL}{\operatorname{GL}}
\DeclareMathOperator{\SL}{\operatorname{SL}}
\DeclareMathOperator{\Aff}{\operatorname{Aff}}
\DeclareMathOperator{\lcm}{\operatorname{lcm}}
\DeclareMathOperator{\Aut}{\operatorname{Aut}}

\begin{document}

\title{Periodic Points of Prym Eigenforms}
\author{Sam Freedman}
\address{Brown University}
\email{sam\_freedman@brown.edu}

\date{November 2022}

\begin{abstract}
A point of a Veech surface is \emph{periodic} if it has a finite orbit under the surface's affine automorphism group.
We show that the periodic points of Prym eigenforms in the minimal strata of translation surfaces in genera 2, 3 and 4 are the fixed points of the Prym involution.
This answers a question of Apisa--Wright and gives a geometric proof of M\"oller's classification of periodic points of Veech surfaces in the minimal stratum in genus 2.
\end{abstract}

\maketitle

\section{Introduction}\label{sec:intro}
A \emph{periodic point} of a Veech surface $M$ is a point $p$ that has finite orbit under the affine automorphism group $\Aff^+(M)$.
Examples of periodic points are the singular points of $M$ and the fixed points of the hyperelliptic involution.
Chowdhury--Everett--Freedman--Lee \cite{icerm} produced an algorithm that inputs a (non-square-tiled) Veech surface and outputs its set of periodic points.
They used this algorithm to conjecture that the periodic points of certain Veech surfaces $S(w, e)$, \emph{Prym eigenforms} in the minimal stratum in genus 3, are the fixed points of an involution.
In this article we prove that conjecture as well as give a uniform classification of periodic points of Prym eigenforms in the minimal strata of genus $g$ translation surfaces $\Omega \mathcal{M}_2(2), \Omega \mathcal{M}_3(4)$ and $\Omega \mathcal{M}_4(6)$:

\begin{theorem}\label{thm:main_thm}
    The periodic points of a nonarithmetic Prym eigenform in the stratum $\Omega\mathcal{M}_g(2g - 2)$ for $2 \le g \le 4$ are the fixed points of its Prym involution.
\end{theorem}

For an overview of the proof, see \autoref{subsec:proof-overview}.
\autoref{thm:main_thm} answers Problem 1.5 in Apisa--Wright \cite{AW} for Prym eigenforms in minimal strata.
The genus 2 case gives a new proof of the classification of periodic points on nonarithmetic Veech surfaces in $\Omega\mathcal{M}_2(2)$, due originally to M\"oller \cite{Moller2006}:
\begin{corollary}
    The periodic points of a nonarithmetic Veech surface in $\Omega\mathcal{M}_2(2)$ are the Weierstrass points.
\end{corollary}

By the classification of the connected components of the Prym eigenform loci due to McMullen \cite{McMullenDiscSpin} and Lanneau--Nguyen \cite{LanneauNguyen} \cite{lanneau_nguyen_2020}, to prove \autoref{thm:main_thm} it suffices to determine the periodic points of one \emph{prototypical eigenform} $M(w, e)$ in each connected component of the loci.
Here $w, e \in \ZZ$ are particular choices of integer \emph{prototypes}; see \autoref{fig:prym-all} for examples in each genus.
\begin{figure}[t]
    \centering
    \includegraphics[scale=0.55]{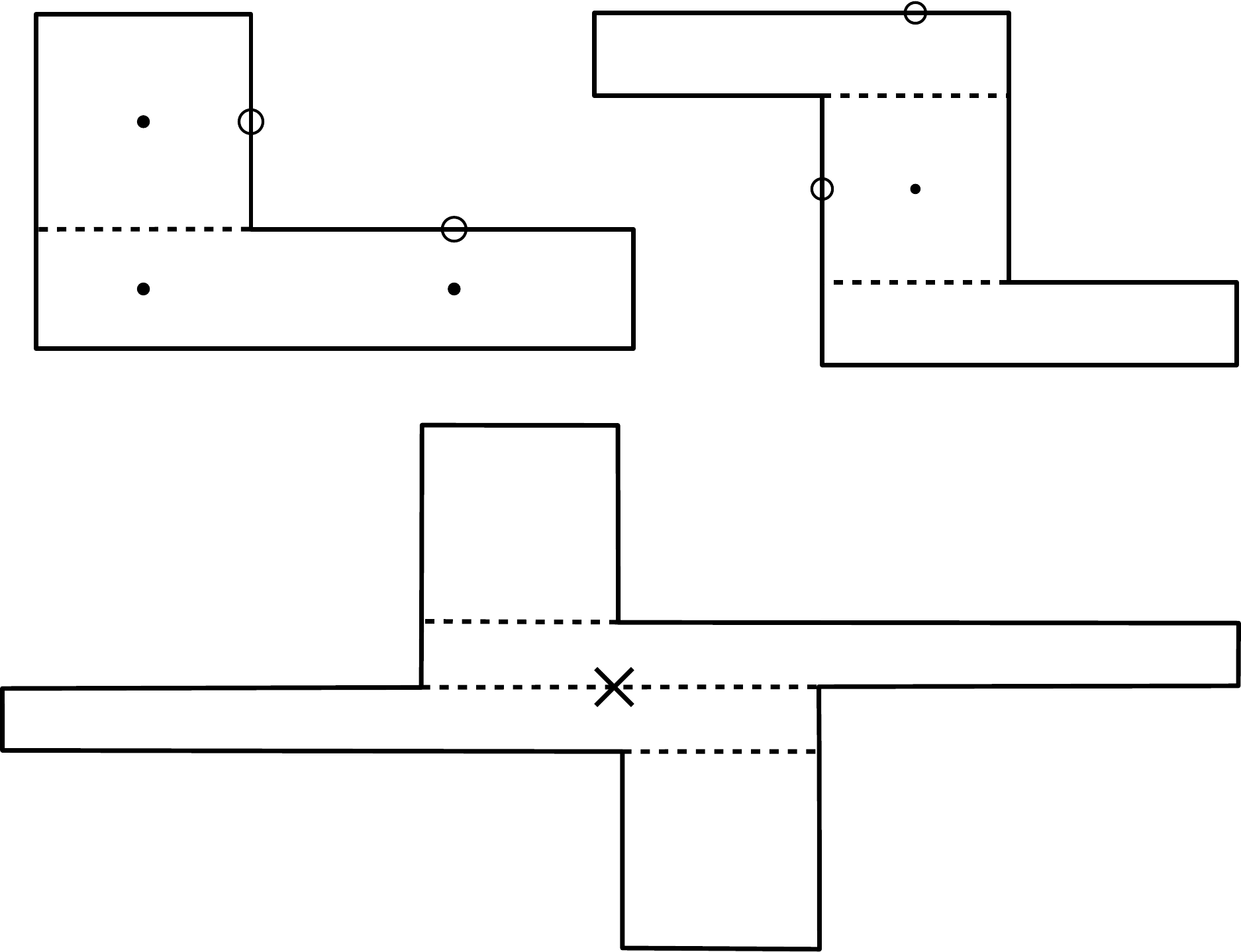}
    \caption{The L-shaped, S-shaped and X-shaped Prym eigenforms considered in this paper.}
    \label{fig:prym-all}
\end{figure}
Although each prototype $M(w, e)$ could a priori require an individual treatment, when the \emph{discriminant}
\[
    D \coloneqq \begin{cases} e^2 + 4w,& g \in {2, 4} \\ e^2 + 8w,& g = 3 \end{cases}
\]
is sufficiently large we give a uniform argument using explicit affine automorphisms called \emph{butterfly moves}.
McMullen \cite{McMullenDiscSpin} first introduced these move to classify the number of connected components of different Teichmüller curves in the strata $\Omega \mathcal{M}_2(2)$.
We use butterfly moves as concrete elements of $\Aff^+(M(w, e))$ to rule out large regions of the surface from containing periodic points.

\subsection*{Applications of \autoref{thm:main_thm}}
Periodic points are central in Teichmüller dynamics because they represent \emph{nongeneric} behavior for the $\GL(2, \RR)$-action on strata of translation surfaces.
Here, \emph{nongeneric} means that the orbit closure of a surface with a marked periodic point $(M, p)$ in a stratum of marked translation surfaces has smaller than expected dimension.
We give three applications here:
\begin{itemize}
    \item \textbf{Counting holomorphic sections of surface bundles over Teichmüller curves}:
    Veech surfaces generate certain surface bundles over (covers of) Teichm\"uller curves.
    (See, e.g., Apisa \cite{ApisaGenus2}.)
    Shinomiya \cite{Shinomiya} proved that a holomorphic section of such bundles is uniquely determined by a choice of periodic point on the generating Veech surface.
    Combining Shinomiya's result with \autoref{thm:main_thm}, we can restrict the number of holomorphic sections of those bundles generated by Prym eigenforms:
    \begin{corollary}
        The surface bundles generated by Prym eigenforms in the minimal strata in genera $2 \le g \le 4$ have $10 - 2g$ distinct holomorphic sections.
    \end{corollary}

    \item \textbf{Solving the finite-blocking problem}:
    Two points $p$ and $q$ on a translation surface $M$ are \emph{finitely blocked} if there is a finite set $B \subset M$ such that all straight line paths from $p$ to $q$ contain a point of $B$.
    Apisa--Wright \cite{AW}*{Theorem 3.6} showed that if $p$ and $q$ are finitely blocked, then either both $p$ and $q$ are periodic points, or $\pi_{Q_{\min}}(p) = \pi_{Q_{\min}}(p)$ where $\pi_{Q_{\min}} : M \to Q_{\min}$ is a certain degree 1 or 2 covering of a half-translation surface (c.f. Apisa--Wright \cite{AW}*{Lemma 3.3}).
    We can then follow the proof of Apisa--Saavedra--Zhang \cite{apisa2020periodic}*{Corollary 1.6} to show
    \begin{corollary}
        The pairs of points on a Prym eigenform in the minimal stratum in genus 2, 3 or 4 that are finitely blocked from each other are a nonsingular point and its image under the Prym involution.
    \end{corollary}

    \item
    \textbf{Evidence for higher-rank orbit closures}:
    Apisa \cite{ApisaGenus2} finished the classification of periodic points on primitive genus 2 translation surfaces that M\"oller \cite{Moller2006} initiated.
    By finding degenerations to lower genus eigenforms with marked points in the boundary, Apisa showed that the gothic locus that Eskin--McMullen--Mukamel--Wright \cite{EMMW} constructed is the unique nonarithmetic rank two affine invariant subvariety in $\Omega\mathcal{M}_4(6)$.
    It seems interesting whether one can combine \autoref{thm:main_thm} with  similar degeneration arguments to classify periodic points on Prym eigenforms in nonminimal strata and restrict the existence of other higher-rank orbit closures.
\end{itemize}

\subsection*{Previous Work}\label{subsec:previous-work}
Gutkin--Hubert--Schmidt \cite{GutkinHubertSchmidt2003} first showed that (nonarithmetic) closed $\GL(2, \RR)$-orbits have a finite number of periodic points (see Chowdhury--Everett--Freedman--Lee \cite{icerm} for another proof).
Eskin--Filip--Wright \cite{efw} then showed finiteness for all nonarithmetic affine invariant subvarieties.

Authors have computed this finite set of periodic points for other Veech surfaces: M\"oller \cite{Moller2006} for Veech surfaces in genus 2, Apisa \cite{ApisaGenus2} for nonarithmetic eigenform loci in genus 2, Apisa \cite{ApisaGL2R} for components of strata of translation surfaces, Apisa--Saavedra--Zhang \cite{apisa2020periodic} for regular $2n$-gons and double $(2n + 1)$-gons, and B. Wright \cite{wright2021periodic} for Veech--Ward surfaces.

\begin{remark}\label{rem:conjectures}
    All known periodic points on Veech surfaces are the fixed points of an involution.
    It would be interesting to see if this holds for the Veech surfaces in the  gothic loci due to Eskin--M\"oller--Mukamel--Wright \cite{EMMW}.

    The philosophy of the algorithm in Chowdhury--Everett--Freedman--Lee \cite{icerm} is that the Rational Height Lemma (see \autoref{lem:rat-ht}) applied with three ``independent'' parabolic directions is enough to determine the periodic points of the surface.
    Our argument in genus two uses three such directions, yet our arguments in genera four and five use slightly more.
    Whether there is another argument using three directions in these latter cases, and formally proving this heuristic, seems worthy of future study.
\end{remark}

\subsection*{Outline of Paper}\label{subsec:paper-outline}
In \autoref{sec:background}, we give background on flat geometry, Prym eigenforms, butterfly moves, and periodic points.

As discussed above, the work of McMullen \cite{McMullenDiscSpin} and Lanneau--Nguyen \cites{LanneauNguyen,lanneau_nguyen_2020} allows us to consider specific choices of Prym eigenforms.
We choose certain prototypes $M(w, e)$ that have evident horizontal and vertical cylinder decompositions, each with a corresponding global multi-twist $T_H$ and $T_V$ respectively.
In \autoref{sec:dehn-twist} we use \autoref{lem:interior-per-pt--mult-one--general} to classify the points of $M(w, e)$ having finite orbit under the subgroup $\langle T_H, T_V \rangle$.
This reduces the problem to considering the periodic points on certain horizontal and vertical boundary saddle connections of $M(w, e)$, which we carry out in \autoref{subsec:g2}, \autoref{subsec:g3} and \autoref{subsec:g4} for genera 2, 3 and 4 respectively.

For this, we use butterfly moves $B_q$ (see \autoref{fig:butterfly-moves} for an example) to produce new cylinder directions on $M(w, e)$ with which we can rule out points from being periodic.
In fact the butterfly move $B_q$ yields a new prototypical eigenform $M'(w, e) \coloneqq B_q \cdot M(w, e)$, onto which we can transfer the candidate periodic points.
One difficulty is that some prototypes $(w, e)$ determine eigenforms $M'(w, e)$ that are not again rectilinear.
But in \autoref{appendix:good-prototypes} we show that, in each connected component of Prym eigenforms, we can find a \emph{good prototype} $M(w, e)$ such that $M'(w, e)$ is rectilinear.
Under the assumption that $M(w, e)$ is good, we can analyze its candidate points on the new prototype $M'(w, e)$ to finish the classification.

\subsection*{Acknowledgements}
Part of this work took place at Univert\'e Paris-Saclay under a Fondation Math\'ematique Jacques Hadamard Junior Scientific Visibility Grant.
We thank Paul Apisa, Jeremy Kahn, Samuel Leli\`evre and Duc-Manh Nguyen for helpful conversations, as well as Vincent Delecroix and Julian R\"uth for their collaboration and support in working with Flatsurf.
We thank Ethan Dlugie, Joseph Hlavinka, Siddarth Kannan and Jordan Katz for comments and discussions on earlier drafts.

\section{Background}\label{sec:background}
\subsection{Flat Geometry}
General surveys on translation surfaces are Zorich \cite{zorich}, Wright \cite{wrightsurv} and Massart \cite{massart}.

A \emph{saddle connection} $\gamma$ on a translation surface is a straight-line locally geodesic segment connecting two cone points and having no cone points in its interior.
Each saddle connection has an associated vector in $\mathbb{C}$ defined up to $\pm 1$ called its \emph{holonomy} vector.

A \emph{cylinder} $C$ is the isometric image of a right Euclidean cylinder having a union of saddle connections for each boundary component.
Cylinders have a foliation by homotopic closed trajectories of the straight-line flow that are parallel to the boundary saddle connections, determining a \emph{direction} for the cylinder.
The cylinder is \emph{simple} if each boundary component consists of a single saddle connection.
For example, the top square horizontal cylinder of \autoref{fig:prym2} is simple, yet the bottom rectangular horizontal cylinder is not.
The total length of its boundary components is its \emph{circumference} $c(C)$, and the perpendicular distance between them is its \emph{height} $h(C)$.
The ratio $m(C) \coloneqq h(C)/c(C)$ is its \emph{modulus}.
A \emph{cylinder decomposition} of $M$ is a collection of parallel cylinders whose union is $M$.

\subsection{Veech Surfaces and Multi-Twists}\label{subsec:multi-twists}
A reference for this section is Hubert--Schmidt \cite{HubertSchmidt}.

An \emph{affine automorphism} of $M$ is a self-diffeomorphism that, in charts from the translation atlas of $M$, has the form $\mathbf{x} \mapsto A\mathbf{x} + \mathbf{b}$ with $A \in \SL(2, \RR)$ and $\mathbf{b} \in \RR^2$; we write $\Aff^+(M)$ for the group of orientation-preserving affine automorphisms.
One checks that the matrix $A$ is independent of choice of coordinates, giving a well-defined map derivative map $D: \Aff^+(M) \to \SL(2, \RR)$.
The \emph{Veech group} $\SL(M)$ is the image of $\Aff^+(M)$ in $\SL(2, \RR)$.
For example, the Veech group of the square torus is $\SL(2, \ZZ)$.

There is a short exact sequence
\[
    1 \to \Aut(M) \to \Aff^+(M) \xrightarrow{D} \SL(M) \to 1,
\]
where the kernel $\Aut(M)$ is the group of \emph{translation automorphisms}, i.e., affine automorphisms having linear part equal to the identity.
Because $\Aut(M)$ is generically trivial (See M\"oller \cite{MollerSurv}*{Theorem 1.1}, and is trivial in all cases considered in this paper, we will identify $\Aff^+(M)$ with $\SL(M)$.

An important class of translation surfaces are $\emph{Veech surface}$: those for which $\SL(M) \le \SL(2, \RR)$ is a lattice (i.e., for which $\HH^2 / \SL(M)$ has finite area).
An important consequence of the lattice property is the \emph{Veech dichotomy}: in every direction $\theta$, either all bi-infinite straight-line trajectories are periodic or uniformly distribute over the surface.

As a consequence, Veech surfaces are \emph{parabolic} in the direction of every saddle connection, meaning that in those directions $M$ decomposes into cylinders having rationally-related (i.e., \emph{commensurable}) moduli.
The commensurability allows the Dehn twists in each cylinder to assemble into a global affine automorphism: a Dehn multi-twist $\phi \in \Aff^+(M)$ having derivative $A \coloneqq D\phi \in \SL(M)$ of form
\[
    A = \begin{pmatrix} 1 & t \\ 0 & 1 \end{pmatrix}, \quad t \coloneqq \lcm(m_1^{-1}, \dots, m_n^{-1}).
\]
See, e.g., McMullen \cite{McMullenEigen}*{Lemma 9.7} for more details.

In the above setup, for each cylinder $C_i$ we can write $t = a_i m_i^{-1}$ for some $a_i \in \NN$.
We will call $a_i$ the \emph{multiplicity} of the cylinder $C_i$ in the surface $M$.
The multiplicity records the number of Dehn twists the global multi-twist $A$ does in the cylinder $C_i$.

\subsection{Prym Eigenforms}\label{subsec:prym-eigenforms}
McMullen \cite{McMullenPrym} introduced Prym eigenforms as generalizations of his construction of Veech surfaces in genus 2 (see McMullen \cite{McMullenEigen}).
They come equipped with a \emph{Prym involution} $i_M: M \to M$ such that $(i_M)^*$ acts as $-I$ on $H^*(M)$, and the quotient $M / i_M$ has genus two less than $M$.
The Riemann-Hurwitz Theorem implies that $i_M$ has $10 - 2g$ fixed points.

McMullen proved that the Prym eigenforms form $\GL(2, \RR)$-invariant \emph{Prym eigenform loci} $\Omega E_D(2g - 2) \subset \Omega\mathcal{M}_g(2g - 2)$ where $D$ is the \emph{discriminant}, an integer parameter congruent to 0 or 1 modulo 4.
The loci $\Omega E_D(2g - 2)$ consist of up to two $\GL(2, \RR)$-orbits (see McMullen \cite{McMullenDiscSpin} and Lanneau--Nguyen \cite{LanneauNguyen} \cite{lanneau_nguyen_2020}).
These authors constructed explicit polygonal \emph{prototypical Prym eigenforms} (or \emph{prototypes}) $M(w, h, t, e)$ generating these closed $\GL(2, \RR)$-orbits; see \autoref{fig:prym2}, \autoref{fig:prym3-prototypes}, and \autoref{fig:prym4-disc53} for examples.
Four integers $(w, h, t, e) \in \ZZ^4$ satisfying certain arithmetic conditions determine the prototype (see \autoref{sec:main_thm});
conversely, McMullen \cite{McMullenDiscSpin} and Lanneau--Nguyen \cite{LanneauNguyen} \cite{lanneau_nguyen_2020} showed that every Prym eigenform is $\GL(2, \RR)$-equivalent to a prototype $M(w, h, t, e)$ for $(w, h, t, e)$ satisfying the conditions.
In this paper we write $M(w, e) \coloneqq M(w, 1, 0, e)$ and $M'(w, e)$ when $h \neq 1$ and/or $t \neq 0$, since we won't have need for the explicit values of $h$ or $t$.

\subsection{Butterfly Moves}\label{subsec:butterfly-moves}
Butterfly moves $B_k, k \in \NN \cup \{\infty\}$ are explicit affine transformations that exhibit two prototypes $M(w, e)$ and $M(w', h', t', e')$ as lying on the same $\SL(2, \RR)$-orbit.
In other words, they ``connect'' the cusp of the Veech group corresponding to the horizontal direction of $L(w, e)$ to the cusp coming from another \emph{butterfly cylinder } direction on the surface.
McMullen \cite{McMullenDiscSpin}*{Section 7} introduced these moves to classify the number of connected components of the eigenform loci $\Omega E_D(2) \subset \Omega\mathcal{M}_2(2)$; Lanneau--Nguyen \cite{LanneauNguyen} \cite{lanneau_nguyen_2020} used analogous moves to study the connected components of the higher genus Prym eigenform loci $E_D(4)$ and $E_D(6)$.
We will use butterfly moves and the fact that $\GL(2, \RR)$ preserves periodic points to transport candidate periodic points of $L(w, e)$ to another prototype $L(w', h', t', e')$ on which we can classify them.

We now describe the details of the construction for a genus 2 eigenform $L(w, e)$ that are relevant for the proof of \autoref{thm:main_thm}, following the exposition in McMullen \cite{McMullenDiscSpin}*{\S 7}:

For an integer $k \in \NN$ satisfying the \emph{admissibility criterion} $(e + 2k)^2 < D$ (see McMullen \cite{McMullenDiscSpin}*{Theorem 7.2}), there is curve $\gamma_k$ contained in the long horizontal cylinder of $L(w, e)$ that wraps once horizontally and $k$ times vertically.
In other words, $\gamma_k$ is a straight-line trajectory having slope $k/w$ that emanates from the lower-left corner of the long horizontal cylinder.
We also let $\gamma_\infty$ denote the curve wrapping once in the vertical direction.
Completing $\gamma_k$ to a cylinder decomposition produces another two-cylinder decomposition of $L(w, e)$ on which $\gamma_k$ bounds a simple cylinder on one side.
We call $\gamma_k$ the \emph{$k$-butterfly curve}, and the resulting simple cylinder $C_k$ the \emph{$k$-butterfly cylinder}.
We also choose a curve $\eta$ that connects the two boundary components of $C_k$.
See \autoref{fig:butterfly-moves} for the case $k = 2$.
\begin{figure}[ht]
    \centering
    \includegraphics[width=\textwidth]{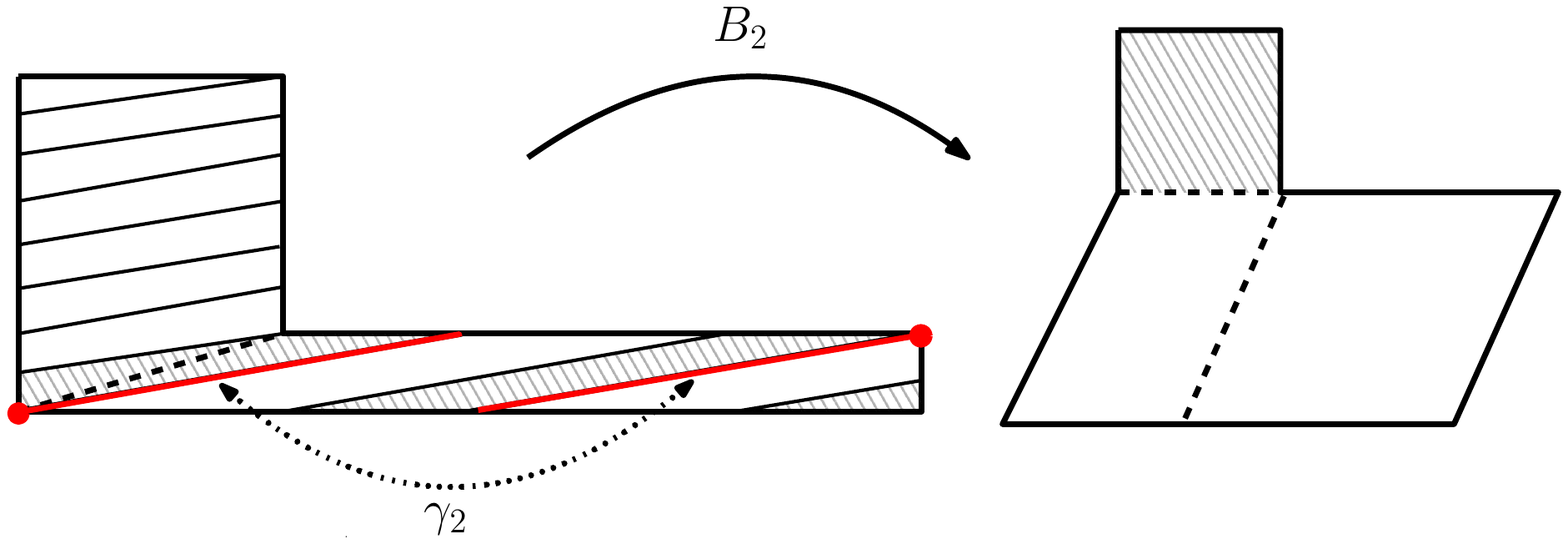}
    \caption{The butterfly move $B_2$ on an L-shaped Prym eigenform.}
    \label{fig:butterfly-moves}
\end{figure}

It follows that whenever $k \in \NN \cup \{\infty\}$ is admissible, we can use $\GL(2, \RR)$ to normalize $C_k$ to a rectilinear square with the images of $\gamma_k$ and $\eta$ being the horizontal and vertical sides.
The \emph{$k$-butterfly matrix} carrying out this normalization is
\begin{equation}
    B_k \coloneqq \begin{bmatrix} w & \lambda \\ k & 1  \end{bmatrix}^{-1}, \quad (k \in \NN).
\end{equation}
McMullen \cite{McMullenDiscSpin}*{Theorem 7.3} then shows that, up to a scaling factor and horizontal shear, the surface $B_k \cdot L(w, e)$ becomes a new prototypical eigenform $L(w', h', t', e')$ for new parameters $(w', h', t', e') \in \ZZ^4$ satisfying the same arithmetic conditions as $(w, 1, 0, e)$.
(Specifically, we have $e' = -e - 8$, $h' = \gcd(2, w)$, $w' = (D - e'^2)/(4w')$, and $t' \in \{0, 1\}$ depending on whether $\gcd(w', h', e') = 1$.)

\subsection{Periodic Points}\label{subsec:periodic-points}
\subsubsection*{Rational Height Lemma}
Our main tool for demonstrating that a point $p \in M$ is not periodic is the \emph{Rational Height Lemma},
It has appeared in, e.g., Apisa \cite{ApisaGL2R}, Apisa--Saavedra--Zhang \cite{apisa2020periodic}, and B. Wright \cite{wright2021periodic}.

Let $h(C, p)$ denote the \emph{height} of point $p$ in a cylinder $C$, that is the perpendicular distance from the bottom boundary saddle connection of $C$ to $p$.
We say $p$ has \emph{rational height} in $C$ if $h(C, p)/h(C)$ is rational.
The Rational Height Lemma says that a point having irrational height in a cylinder with a direction that is parabolic on $M$ cannot be periodic:
\begin{lemma}[Rational Height Lemma]\label{lem:rat-ht}
    If $C$ is a cylinder having a parabolic direction and $p \in C$ is a periodic point, then $p$ has rational height in $C$.
\end{lemma}
\begin{proof}
    By rotating the surface, we may assume that $C$ is horizontal.
    Set $h \coloneqq h(C)$ and $c \coloneqq c(C)$.
    Choose flat coordinates on $C$ so that the origin is on the lower boundary component and write $p = (x, y)$.
    We know that the global parabolic multi-twist has the form $T_C = \begin{bmatrix} 1 & n \frac{c}{h} \\ 0 & 1 \end{bmatrix}$ for some nonzero $n \in \NN$.

    Because $p$ is periodic, its orbit under the subgroup $\langle T \rangle$ is finite.
    We can then find distinct natural number $k > \ell \ge 0$ such that $T^k_C(p) = T^\ell_C(p)$.
    This implies that
    \[
        ([x + k \cdot \frac{nc}{h}y] \pmod{c}, y) = ([x + \ell \cdot \frac{nc}{h}y] \pmod{c}, y),
    \]
    or $k \cdot \frac{nc}{h}y - \ell \cdot \frac{nc}{h}y = ac$ for some $a \in \ZZ$.
    Dividing through by $c$ yields $y/h n (k - \ell) = a$, from which we conclude that $y / h \in \QQ$.
\end{proof}

\subsubsection*{Prym fixed points}
As mentioned in the introduction, fixed points of the hyperelliptic involution are important examples of periodic points on Veech surfaces.
Via an analogous argument, we see that fixed points of the Prym involution are also periodic:
\begin{lemma}\label{lem:prym-fps-periodic}
    The fixed points of the Prym involution are periodic points.
\end{lemma}
\begin{proof}
    (Following Apisa--Saavedra--Zhang \cite{apisa2020periodic}*{Remark 2.9})
    Since $D(i_M) = -I \in \SL(M)$ and $\Aff^+(M) \cong \SL(M)$ for $M$ a Prym eigenform, $i_M$ commutes with all elements of $\Aff^+(M)$.
    Then for $p$ a fixed point of the Prym involution, for any $\phi \in \Aff^+(M)$ we have
    \[
        i_M (\phi(p)) = \phi(i_M(p)) = \phi(p),
    \]
    and $\phi(p)$ is also a fixed point.
\end{proof}

\begin{remark}\label{rem:alternate-definition-of-per-pt}
    Apisa \cite{ApisaGL2R} gave a definition of periodic point for general translation surfaces that specializes to ours when the surface is Veech.
    The condition for $p \in M$ being periodic becomes
    \begin{equation*}
        \dim_{\CC} \overline{\SL(2, \RR) \cdot (M, p)} = \dim_{\CC} \overline{\SL(2, \RR) \cdot M}.
    \end{equation*}
    (The left-hand side denotes the orbit closure in a stratum of marked translation surfaces.)
\end{remark}

\section{Periodic Points of Two Multi-twists}\label{sec:dehn-twist}
In this section, we classify the points having finite order under the subgroup $\langle T_H, T_V \rangle \subset \SL(M)$ generated by the horizontal and vertical multi-twist.
(We use the horizontal and vertical directions for convenience; the Lemmas hold for any two transverse parabolic directions on $M$ by first applying a shear.)

We first classify the periodic points interior to a horizontal multiplicity 1 cylinder that decomposes into two rectangular \emph{regions}, i.e., a connected component of an intersection between a horizontal cylinder and a vertical cylinder.
Recall (see \autoref{subsec:multi-twists}) that a multiplicity 1 cylinder is one in which the global horizontal multi-twist does a single Dehn twist.

\begin{lemma}\label{lem:two-regions-mult-one}
    Let $M$ be a horizontally and vertically parabolic translation surface containing a horizontal cylinder $H$ of multiplicity 1.
    Suppose that two vertical cylinders $V_1$ and $V_2$ having irrational ratio of heights $h(V_1)/h(V_2)$ divide $H$ into two rectangular regions $R_1$ and $R_2$.
    Then, if there is a periodic point in the interior of $H$, it must lie at the center of $R_1$ or $R_2$.
\end{lemma}
\begin{proof}
    We give the argument for the region $R_1$ as the case of $R_2$ follows by symmetry.
    After a horizontal and vertical rescaling, we may assume that $R_1$ has height and width 1, meaning that the other region $R_2$ has some irrational height $\alpha$.
    Because $H$ has multiplicity 1, the Dehn twist $T_H \in \SL(M)$ has the form
    \[T_H \coloneqq \begin{pmatrix} 1 & 1 + \alpha \\ 0 & 1 \end{pmatrix}.\]
    Choose Euclidean coordinates on $H$ with the lower-left corner of $R_1$ being the origin, and consider a periodic point $p = (x, y) \in R_1$.
    The Rational Height Lemma \ref{lem:rat-ht} applied to $H$ and $V_1$ implies that both $x$ and $y$ are rational.

    We claim that if $p$ is a periodic point interior to $H$, it must lie along the diagonal segment $y = -x + 1$.
    There are three possible cases for the image of $p$ under the Dehn twist $T_H$, as illustrated in \autoref{fig:two-dts}:
    \begin{figure}[ht]
        \centering
        \includegraphics[scale=0.7]{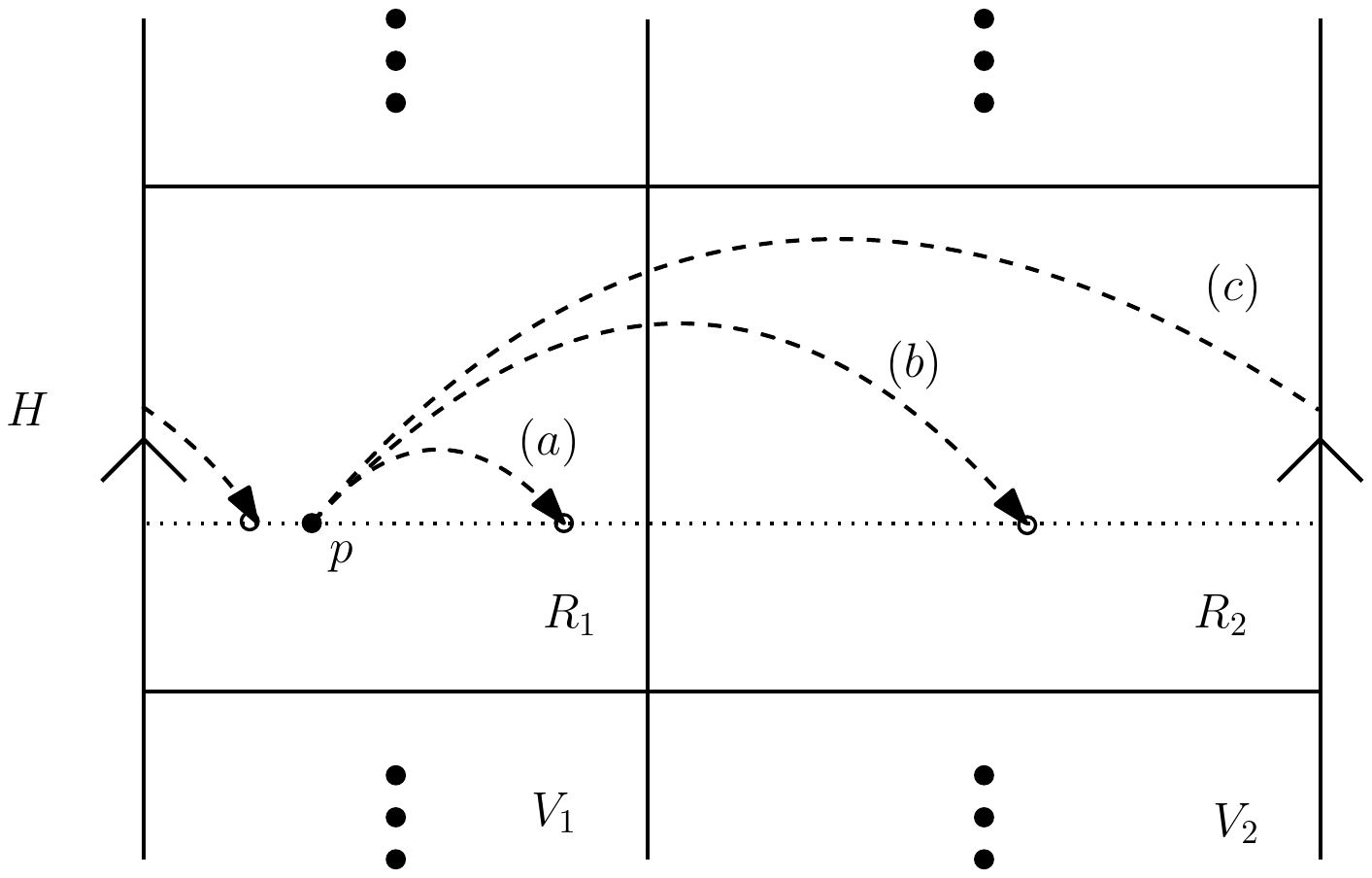}
        \caption{The three possibilities for the image of $p$ under the horizontal twist $T_H$.}
        \label{fig:two-dts}
    \end{figure}

    \begin{itemize}
        \item[(a)] \textbf{The point $p$ remains in $R_1$ under $T_H$:}
        Since the coordinates of $T_H(p)$ are $(x + (1 + \alpha)y, y)$, the height of $T_H(p)$ in $V_1$ is $x + (1 + \alpha)y$.
        The Rational Height Lemma \ref{lem:rat-ht} in $V_1$ implies that $x + (1 + \alpha)y$ is rational.
        Since $x$ and $y$ are rational, $\alpha y$ must also be rational.
        But $\alpha \notin \QQ$ and $y \in \QQ$, so we must have $y = 0$ and $p$ is not an interior point.
        This contradicts our hypotheses, so this case does not occur.

        \item[(b)] \textbf{The point $p$ moves to $R_2$ under $T_H$:}
        Since $T_H(p)$ has height $[x + (1 + \alpha)y] - 1$ in $V_2$, the Rational Height Lemma \ref{lem:rat-ht} in $V_2$ implies that $(x + (1 + \alpha)y) - 1)\alpha^{-1}$ is rational.
        Since $y$ is rational, we also have that $(x + y - 1)\alpha^{-1}$ is rational.
        But $x + y - 1$ is rational and $\alpha^{-1}$ is irrational by assumption, so we conclude $x + y - 1 = 0$.

        \item[(c)] \textbf{$p$ wraps around $H$ through $R_2$ and returns to $R_1$:}
        In this case, $T_H(p)$ has height $[x + (1 + \alpha)y] - (1 + \alpha)$ in $V_1$.
        The Rational Height Lemma in $V_1$ implies that $x + (1 + \alpha)y - (1 + \alpha)$ is rational.
        This means $(1 + \alpha)(y - 1)$ is rational, but $(y - 1)$ is rational and $(1 + \alpha)$ is irrational.
        We conclude that we must have $y = 1$, and $p$ is not an interior point.
        This again contradicts our hypotheses, so this case does not occur.
    \end{itemize}
    We conclude that an interior periodic point $p$ satisfies the equation $y = -x + 1$.

    We can now repeat the argument with the inverse horizontal multi-twist $T_H^{-1}$.
    Reasoning as above, the point $p$ is periodic precisely when $T_H^{-1}$ lies in $R_2$.
    The Rational Height Lemma implies $[x - (1 + \alpha)y]\alpha^{-1} \in \QQ$.
    This means $(x - y)\alpha^{-1}$ is rational, forcing $x = y$.

    Since $p$ lies on $y = -x + 1$ and also $y = x$, we conclude that $p$ lies at the center of $R_1$.
\end{proof}

We require a more general form of \autoref{lem:two-regions-mult-one} that allows for $H$ to consist of more than two regions.
Recall that two parallel cylinders $C_1$ and $C_2$ in a parabolic direction are by definition \emph{commensurable}, i.e., their moduli $h(C_1)/c(C_1)$ and $h(C_2)/c(C_2)$ are rationally related.
If we have the stronger condition that their heights $h(C_1)$ and $h(C_2)$ are also rationally related, then we will say that $C_1$ and $C_2$ are \emph{strongly commensurable}.

\begin{lemma}\label{lem:interior-per-pt--mult-one--general}
Let $M$ be a horizontally and vertically parabolic translation surface having a horizontal cylinder $H$ with multiplicity 1.
Let $\mathcal{V}_1$ and $\mathcal{V}_2$ be two distinct families of strongly commensurable vertical cylinders such that
\begin{itemize}
    \item For $i \in \{1, 2\}$, the intersection of $\bigcup_{V \in \mathcal{V}_i} \overline{V}$ with the interior of $H$ is a (connected) rectangle $R_i$ , and
    \item $\overline{H} = R_1 \cup R_2$.
\end{itemize}
Then, any periodic point in the interior of $H$ must lie at the center of $R_1$ or $R_2$.
\end{lemma}
The proof follows that of \autoref{lem:two-regions-mult-one}.
\begin{remark}\label{rem:lemma-appears-in-apisa}
    Part of the above Lemma appears in Apisa \cite{ApisaGL2R}*{Lemma 5.5} for \emph{generic} translation surfaces, i.e., those with $\GL(2, \RR)$-orbits dense in their strata.
\end{remark}

\section{Proof of \autoref{thm:main_thm}}\label{sec:main_thm}
\subsection{Overview of the proof}\label{subsec:proof-overview}
Because we can use the algorithm of Chowdhury--Everett--Freedman--Lee \cite{icerm} to compute the periodic points of any prototypical eigenform $M(w, e)$, it suffices to assume that $D$ is sufficiently large.
(Our argument is effective in that it computes the specific lower bounds on $D$.)

For such a prototype $M(w, e)$, we first classify the periodic points that lie in the interior of a horizontal or vertical cylinder.
The classification leverages that a periodic point on a translation surface lies at rational height in any parabolic cylinder decomposition containing it (see the Rational Height Lemma, \autoref{lem:rat-ht}).
In genus 2 and 3 the above suffices to classify all periodic points on the interiors of these cylinders, whereas in genus 4 we require an ad hoc argument to rule out eight more interior points (see \autoref{lem:g4-remove-interior-per-pt}).

The remaining locations that could contain periodic points are the boundary saddle connections of the cylinders.
Because the discriminant $D$ is sufficiently large, we can construct a new cylinder decomposition with the butterfly move $B_2$ (see \autoref{subsec:butterfly-moves}).
After applying the Rational Height Lemma in this new decomposition, we find that the midpoint of each boundary saddle connection is its lone candidate periodic point.

While in genus 2 those midpoints are periodic points, in genera 3 and 4 we must rule out some of them as being periodic.
For this we track the midpoints under a particular butterfly move $B_q$, where $q \in \{2, 4, 8\}$ is a function of the prototype $(w, e)$.
But for the same hypotheses on $M(w, e)$ to hold on $B_q(M(w, e))$, i.e. that $B_q(M(w, e))$ has \emph{zero tilt} or rectilinear, certain arithmetic conditions on the integers $(w, e)$ must hold.
In \autoref{appendix:good-prototypes}, we show that every connected component of the Prym eigenform locus $\Omega E_D(4)$ and $\Omega E_D(6)$ contain such a \emph{good prototype}.
After applying the move $B_q$, we can conclude that midpoints are not periodic, completing the proof.

\subsection{Genus 2}\label{subsec:g2}
For positive $w \in \ZZ$ and $e \in \ZZ$, consider the L-shaped translation surface $L(w, e)$ as shown in \autoref{fig:prym2}.
It consists of a $\lambda \times \lambda$ square attached to a $w \times 1$ rectangle, where $\lambda = (e + \sqrt{e^2 + 4w})/2$.
\begin{figure}[ht]
    \centering
    \includegraphics[scale=0.8]{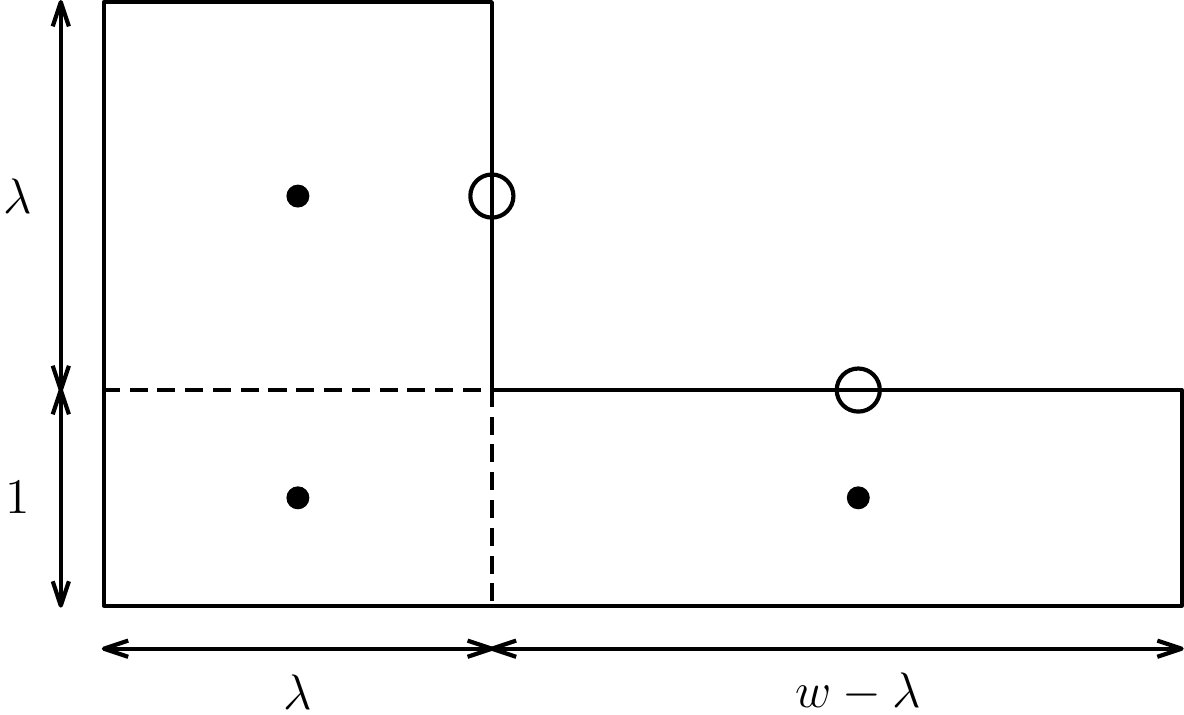}
    \caption{The L-shaped Prym eigenform $L(w, e)$ with $\lambda = (e + \sqrt{D})/2$. The marked points are the Weierstrass points.}
    \label{fig:prym2}
\end{figure}
The surface $L(w, e)$ decomposes into two horizontal cylinders: a \emph{long} cylinder of multiplicity 1 and a \emph{short} simple cylinder of multiplicity greater than 1.
McMullen \cite{McMullenEigen} proved that for certain choices of $w$ and $e$, $L(w, e)$ is a Veech surface generating a Teichm\"uller curve of discriminant $D = e^2 + 4w$.
He \cite{McMullenDiscSpin} also classified the connected components of the eigenform loci: the locus $\Omega E_D(2)$ has one component when $D\not\equiv 1 \pmod{8}$ and has two components when $D\equiv 1 \pmod{8}$.
Given this classification, it suffices to consider the Prym eigenforms $L(w, e)$ with parameters $(w, e)$ equal to
\begin{itemize}
    \item $((D - 4) / 4, -2)$ when $D\equiv 0 \pmod{4}$,
    \item $((D - 1)/4, -1)$ when $D\equiv 5 \pmod{8}$, and
    \item $((D - 1)/4, -1)$ and $((D - 9)/4, -3)$ when $D\equiv 1 \pmod{8}$.
\end{itemize}

\begin{proof}[Proof of \autoref{thm:main_thm} in genus 2]
When $D > 17$, there is a prototypical surface $L(w, e)$ for which the butterfly move $B_2$ is admissible and yields a cylinder decomposition of the same combinatorial type (see McMullen \cite{McMullenDiscSpin}*{Theorem 7.2}).
In the remaining cases when $D \in \{5, 12, 17\}$, then we use the algorithm of Chowdhury--Everett--Freedman--Lee \cite{icerm} to compute the periodic points.

Applying \autoref{lem:two-regions-mult-one} with the long horizontal cylinder and the long vertical cylinder, we find that the potential interior periodic points are the solid dots in \autoref{fig:prym4-disc53}.
The Prym involution fixes all three of these points, so they are in fact periodic.

These two cylinders cover the entire surface except for the long horizontal saddle connection of the long horizontal cylinder and the long vertical saddle connection of the long vertical cylinder.
We first show that the unique periodic point along the long horizontal saddle connection is its midpoint.
As the butterfly move $B_2$ is admissible, there is another two-cylinder decomposition in the direction of the slope $2/w$.
See \autoref{fig:g2-b2-cylinder-decompositions} for two examples.
\begin{figure}[ht]
    \centering
    \includegraphics[scale=0.4]{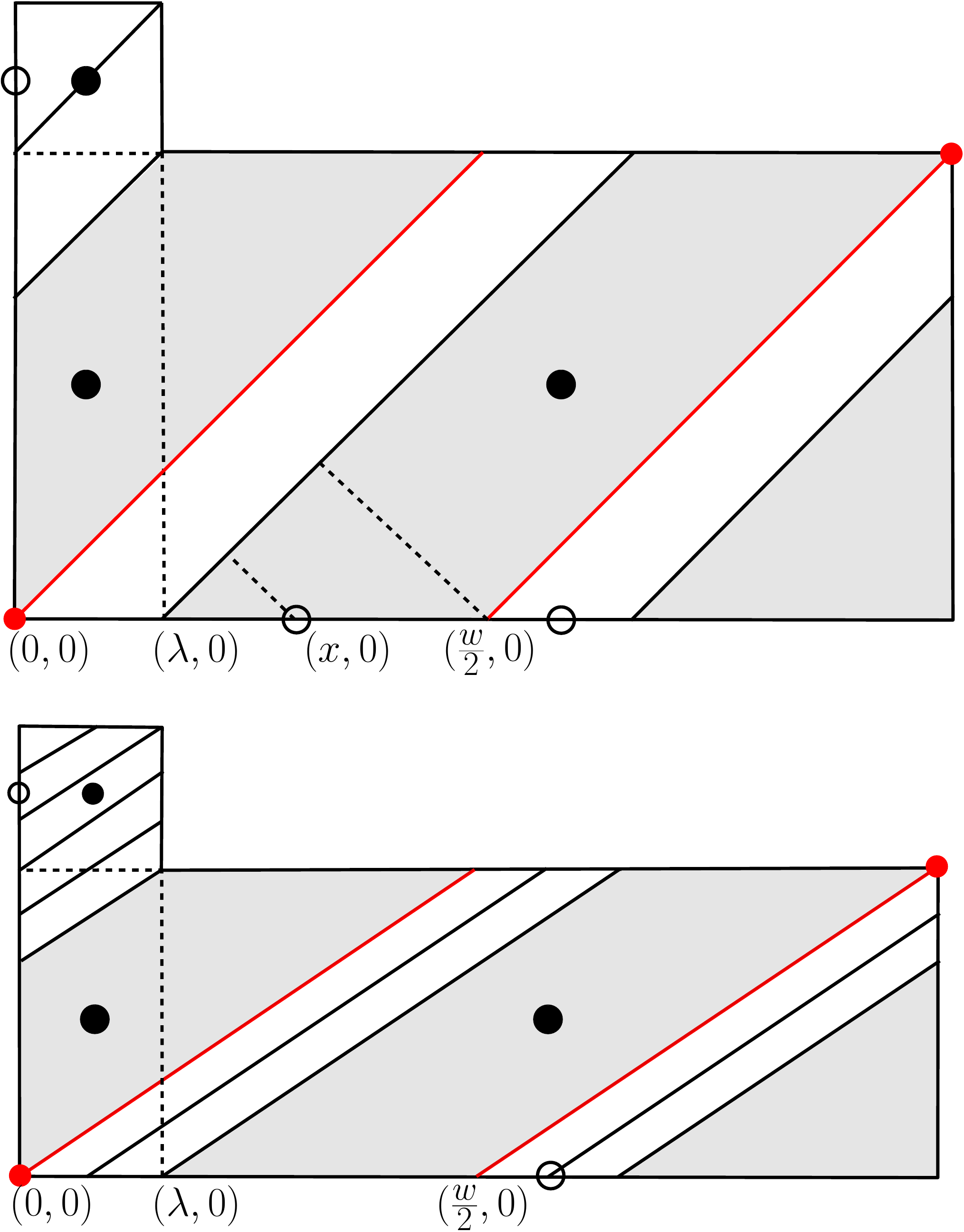}
    \caption{Cylinder decompositions for the eigenforms $L(2, -6)$ and $L(3, -6)$ in the direction of the butterfly move $B_2$.
    The parity of $w$ determines the winding of the long butterfly cylinder.}
    \label{fig:g2-b2-cylinder-decompositions}
\end{figure}
The butterfly saddle connections cuts the long horizontal saddle connection into either three or four components depending on the parity of $w$.

Now fix a periodic point $P = (x, 0)$ along the long horizontal saddle connection, taking the origin as the lower left corner of the long horizontal cylinder.
See \autoref{fig:g2-b2-cylinder-decompositions}.
By applying the Prym involution, we can assume that $\lambda \le x \le (w - \lambda)/2$.
The Rational Height Lemma (\autoref{lem:rat-ht}) in the vertical short cylinder implies that $(x - \lambda) / (w - \lambda)$ is rational, i.e., $x - \lambda = q(w - \lambda)$ for some $q \in \QQ$.

The point $p$ lies in either the short butterfly cylinder or the long butterfly cylinder.
In the first case, the Rational Height Lemma and similar triangles implies
\begin{equation*}
    \frac{x - \lambda}{w/2 - \lambda} = \frac{q(w - \lambda)}{w/2 - \lambda} = \frac{q(w/2 - \lambda) + q w/2}{w/2 - \lambda} = q + q \frac{w/2}{w/2 - \lambda}
\end{equation*}
is rational.
The unique solution is $q = 0$, i.e., $p$ is the left-hand endpoint of the horizontal saddle connection.

In the second case, when $w$ is even we have
\begin{equation}\label{eq:g2--rat-ht--int-long-cyl}
    \frac{x - w/2}{\lambda} = \frac{[\lambda + q(w - \lambda)] - w/2}{\lambda} = 1 - q + (q - 1/2)\frac{w}{\lambda}
\end{equation}
is rational.
The unique solution is $q = 1/2$, meaning $p = (\lambda + (w - \lambda)/2, 0)$ is the fixed point on the long horizontal saddle connection.
When $w$ is odd, we replace the denominator of \autoref{eq:g2--rat-ht--int-long-cyl} with $\lambda/2$ and get the same conclusion.

We conclude that the midpoint of the long horizontal saddle connection is the sole point that can be periodic, which holds because it's a fixed point of the Prym involution.

We can now apply the butterfly move $B_\infty$ (see \autoref{subsec:butterfly-moves}) and repeat the same calculations in the new horizontal direction, noting that the move $B_2$ is again admissible on $B_\infty(w, e)$.
\end{proof}

\subsection{Genus 3}\label{subsec:g3}
Consider the $S$-shaped translation surface $S(w, e)$ shown in \autoref{fig:prym3-prototypes}.
\begin{figure}[t]
    \centering
    \includegraphics[scale=0.6]{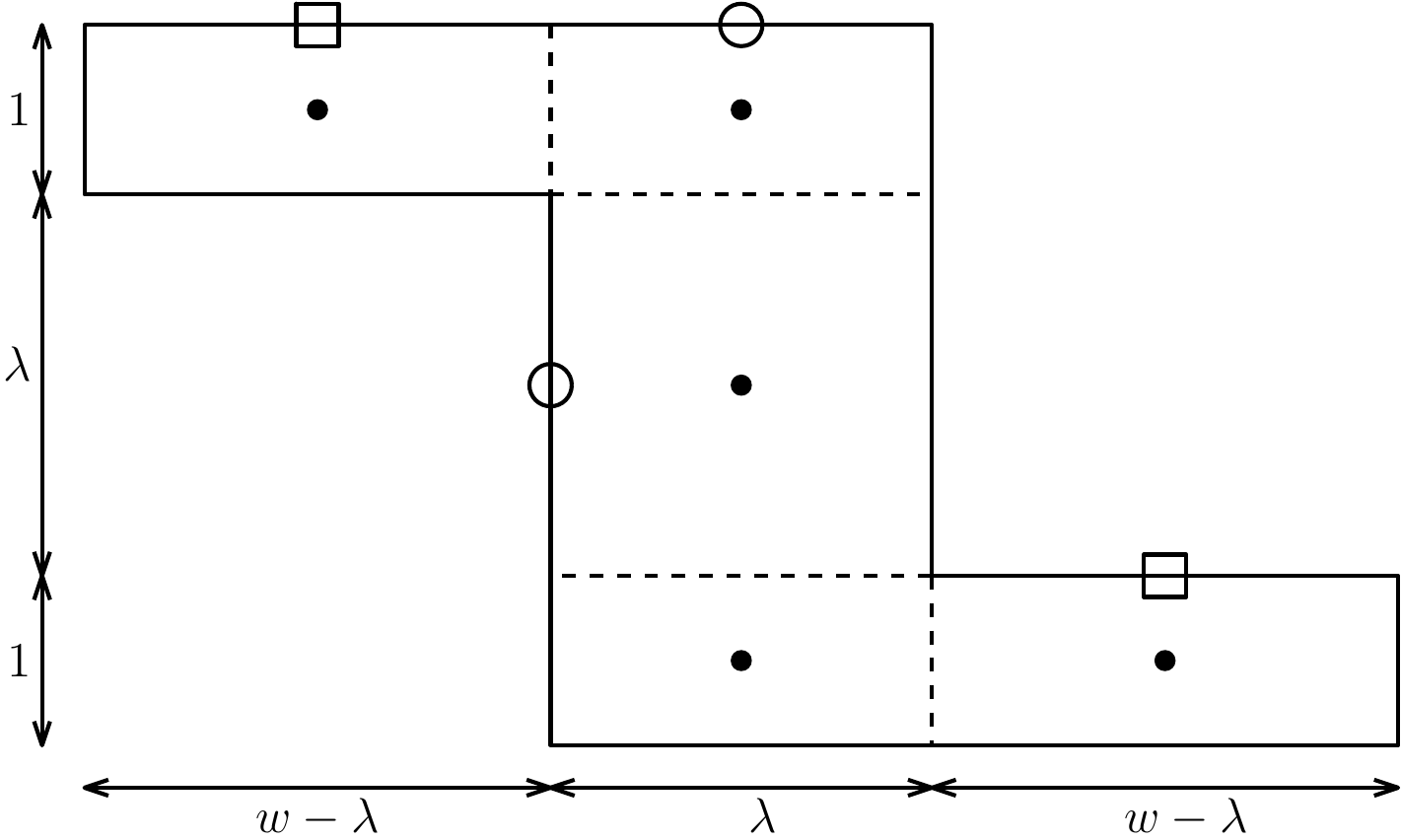}
    \caption{The prototypical surface $S(w, e)$. In contract with genus 2, the points marked with squares are not periodic.}
    \label{fig:prym3-prototypes}
\end{figure}
It consists of a $\lambda \times \lambda$ square attached to two $w \times 1$ rectangles, where now $\lambda = (e + \sqrt{e^2 + 8w})/2$.
In the horizontal direction the two $w \times 1$ form two \emph{long} cylinders and the central square forms a \emph{short} simple cylinder.
It follows that the global horizontal multi-twist has the form $\begin{bmatrix} 1 & w \\ 0 & 1 \end{bmatrix}$, the long cylinders have multiplicity one, and the short cylinder has multiplicity $w$.
Rotating and rescaling $S(w, e)$ with $B_\infty$, we see that $S(w, e)$ has one \emph{long} vertical cylinder of multiplicity one and two \emph{short} simple vertical cylinders with multiplicities greater than one.
McMullen \cite{McMullenPrym} showed that for certain values of $w$ and $e$, the surface $S(w, e)$ is a Prym eigenform in $\Omega E_D(4)$ that generates a Teichmüller curve of discriminant $D = e^2 + 8w$.
(There is an exception for discriminant $D = 8$ which requires a different polygonal model.)

We now recall the classification of connected components of $\Omega E_D(4)$ due to Lanneau--Nguyen \cite{LanneauNguyen}.
When $D < 17$, the two nonempty eigenform loci are $\Omega E_8(4)$ and $\Omega E_{12}(4)$, and they are both connected.
For higher discriminants the locus $\Omega E_D(4)$ is nonempty precisely when $D\equiv 0, 1, \text{or}\ 4\pmod{8}$.
It has two components when $D\equiv 1\pmod 8$ and one component otherwise.

We call a prototype $(w, e)$ \emph{good} if one of the following three situations occurs:
\begin{itemize}
    \item We have $w \equiv 1 \pmod2$ and the butterfly move $B_2$ is admissible.
    \item We have $w \equiv 2 \pmod4$ and the butterfly move $B_4$ is admissible.
    \item We have $w \equiv 4 \pmod8$ and the butterfly move $B_8$ is admissible.
\end{itemize}
Let $B_q$ denotes the admissible butterfly move in each case.
One reason why such $(w, e)$ are good is that resulting prototypes $B_q(w, e) \coloneqq (w', h', t', e')$ have no tilt, i.e., $t' = 0$.
For the calculation, see \autoref{appendix:g3--good-prototypes}.

We claim that for all discriminants $D$ outside an explicit finite set, we can assume our prototypical eigenform $S(w, e)$ has a good prototype:
\begin{lemma}\label{lem:g3--good-prototypes}
    Let $P_{\text{bad}} \coloneqq \lbrace 17$, $20$, $24$, $28$, $32$, $33$, $40$, $41$, $48$, $52$, $56$, $68$, $80$, $84$, $96$, $112$, $128$, $132$, $164$, $228$, $260$, $292$, $388$, $452 \rbrace$.
    For each discriminant $D \notin S_{\text{bad}}$, every connected component of $\Omega E_D(4)$ contains a good prototype $S(w, e)$.
\end{lemma}
For a proof, see \autoref{appendix:g3--good-prototypes}.

\begin{proof}[Proof of \autoref{thm:main_thm} in genus 3]
When $D \in P_{\text{bad}}$, then we use the algorithm of Chowdhury--Everett--Freedman--Lee \cite{icerm} to compute the periodic points on some eigenform $S(w, e)$ in each connected component of $\Omega_D(4)$.
Otherwise, let $S(w, e)$ be a good prototypical eigenform, and let $q \in \{2, 4, 8\}$ denote the integer for which $w \equiv q/2 \pmod{q}$.

Applying \autoref{lem:interior-per-pt--mult-one--general} to the long horizontal cylinders, we see that the interior points of these cylinders that can be periodic are the four points marked with solid dots.
Similarly, applying \autoref{lem:interior-per-pt--mult-one--general} to the long vertical cylinder, we see that the interior points of these cylinders that can be periodic are the two points marked with open dots.
As all four of the open dots move into the long horizontal cylinder under a vertical multi-twist, we conclude that none of them are periodic.
In particular, the periodic points in the interiors of these 3 cylinders are the fixed points of the Prym involution.

The 3 long cylinders cover the entire surface except for the long horizontal saddle connection in the long horizontal cylinders and the long vertical saddle connection in the long vertical cylinder.
First consider the long horizontal saddle connection of the bottom long cylinder.
Because the butterfly move $B_q$ is admissible by hypothesis, the move $B_2$ is also admissible.
As in the genus 2 case we can also decompose the surface by saddle connections with slope $2/w$.
See \autoref{fig:g3-b2-cylinder-decompositions} for two examples.
\begin{figure}[ht]
    \centering
    \includegraphics[scale=.8]{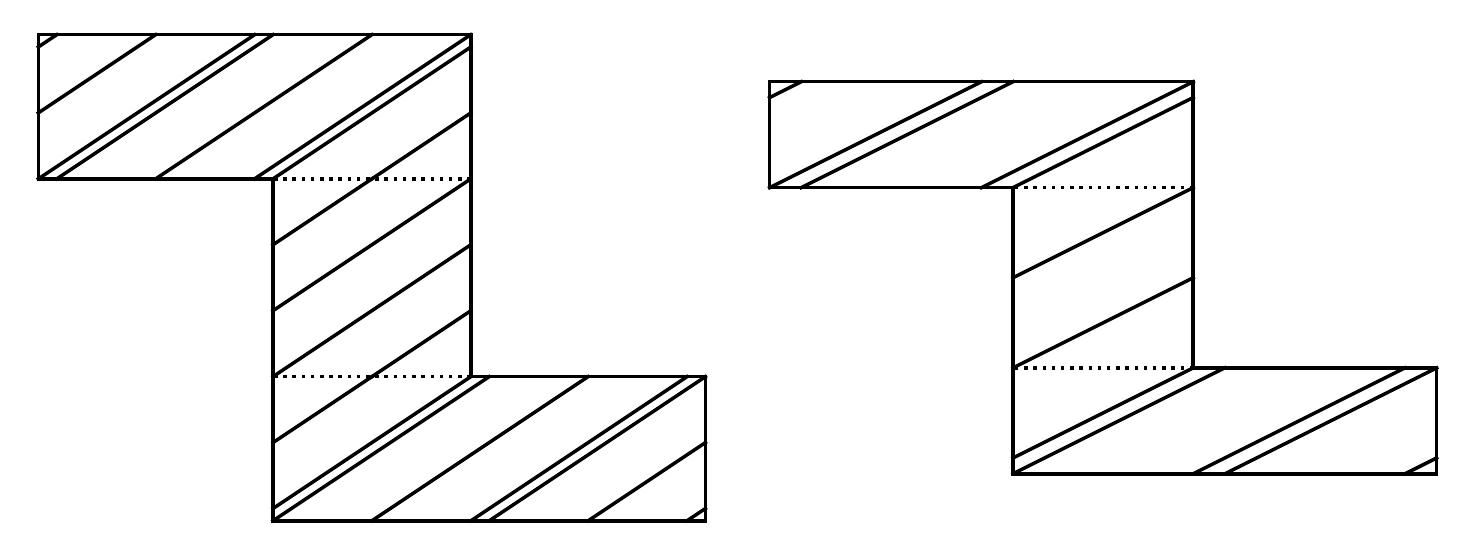}
    \caption{Cylinder decompositions for the eigenforms $S(3, -3)$ and $S(4, -3)$ in the direction of the butterfly move $B_2$.}
    \label{fig:g3-b2-cylinder-decompositions}
\end{figure}

Identical computations as in the genus 2 case show that the midpoint $m$ of the horizontal saddle connection is the sole point that can be periodic.
We now show that $m$ is in fact not periodic.
\begin{lemma}
    The midpoint $m$ lies on a saddle connection $\gamma$ that is the boundary of a $B_q$ butterfly cylinder and passes through a fixed point $s$ of the Prym involution.
\end{lemma}
\begin{proof}
    Recall that the butterfly cylinders have slope $q/w$.
    Consider the geodesic $\gamma$ with slope $q/w$ that starts at the bottom fixed point $s$ in the interior of the long vertical cylinder; see \autoref{fig:g3--fake-wp} for an example.
    \begin{figure}[ht]
        \centering
        \includegraphics[width=\textwidth]{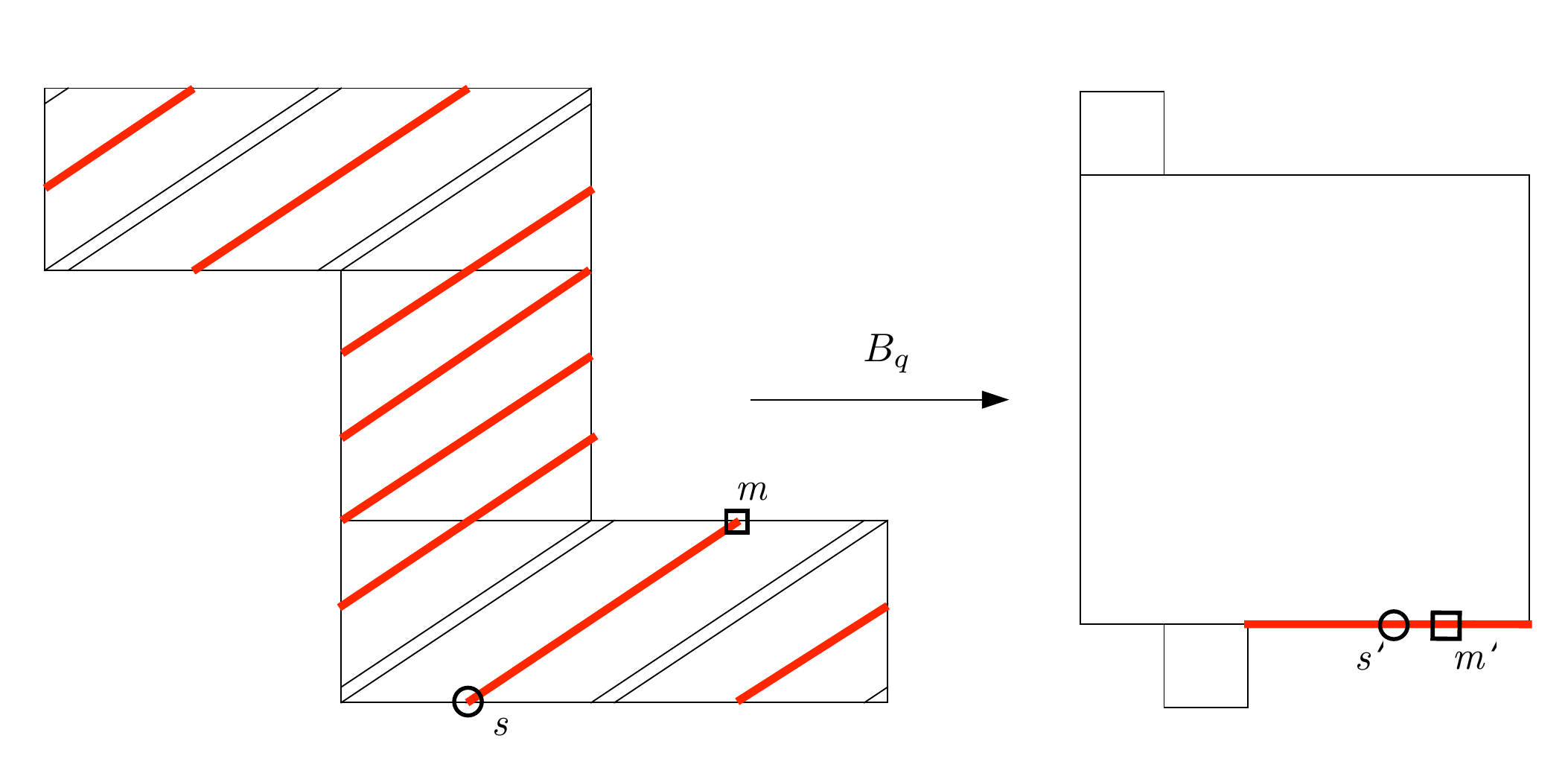}
        \caption{The butterfly move $B_q$ maps $m$ to a point having irrational height on $B_q(w, e)$.}
        \label{fig:g3--fake-wp}
    \end{figure}
    In flat coordinates with the origin at the lower left corner of the long horizontal cylinder, the fixed point has coordinates $(\lambda/2, 0)$ and $m$ has coordinates $(\lambda + (w - \lambda)/2, 0)$.
    Note that $\gamma$ wraps $q/2$ times in the long horizontal cylinder, passes through $m$, then wraps $q/2$ more times before reaching the point $(\lambda/2, 1)$ at the boundary of the central $\lambda \times \lambda$ square.

    We claim that there is a $k \in \NN$ for which $\frac{\lambda}{\lambda/2 + k\lambda} = \frac{q}{w}$, and from this it follows that $\gamma$ wraps $k$ times in the central square before ending at its upper-right vertex.
    To prove the claim, first write $w = q/2 + kq$ for $k \in \NN$ since $(w, e)$ is good.
    Then
    \[
        \frac{q}{w} = \frac{q}{q/2 + kq} = \frac{1}{1/2 + k} = \frac{\lambda}{\lambda/2 + k \lambda},
    \]
    and we're done.
\end{proof}
Let $v$ be the holonomy vector $\begin{bmatrix} w/2 & q/2 \end{bmatrix}^T$ having slope $q / w$ based at the fixed point $s$ with tip at $m$.
Because $v$ maps to (a segment of) the holonomy vector of a horizontal saddle connection, under $B_q$ we compute that $v$ maps to
\begin{align*}
    \frac{\lambda'}{2} \begin{bmatrix} w & \lambda \\ q & 1 \end{bmatrix}^{-1} \begin{bmatrix} w/2 \\ q/2 \end{bmatrix} = \frac{\lambda'}{2w - 2q \lambda}\begin{bmatrix} 1 & -\lambda \\ -q & w \end{bmatrix} \begin{bmatrix} w/2 \\ q/2 \end{bmatrix} = \dots = \begin{bmatrix} \lambda'/4 \\ 0 \end{bmatrix}.
\end{align*}
This implies that $m$ maps $\lambda'/4$ units right of the corresponding fixed point $s'$ on the surface $B_q(w, e)$.
This point has height $(w' - \lambda')/2 + \lambda'/4 = w'/2 - \lambda'/4$ in the new short vertical cylinder having height $w' - \lambda'$.
Since
\[
    \frac{w'/2 - \lambda'/4}{w' - \lambda'} = 1/2 + \frac{\lambda'/4}{w' - \lambda'} \notin \QQ,
\]
we see that $m$ maps to a point that is not periodic and hence itself is not periodic.

It remains to consider the vertical saddle connection of the long horizontal cylinder.
Recall that the butterfly move $B_\infty$ rescales and rotates the surface by $\pi / 2$, creating a new prototype $(w_\infty, e_\infty) \coloneqq (w - e - 2, -e - 4)$.
We claim that $B_2$ is always admissible for $B_\infty(w, e)$.
Indeed we have $(e + 4q)^2 < D$ since $(w, e)$ is good, so that
\[
    ((-e - 4) + 4q)^2 = (e + 4 - 4q)^2 < (e + 4 \cdot 2)^2
\]
and $B_2$ is admissible.
We can then apply a $B_2$ butterfly move on $B_\infty(w, e)$ and repeat the same rational height calculations as for the long horizontal saddle connection.
These calculations show that the sole point along the long vertical saddle connection that is potentially periodic is its midpoint.
This point is in fact periodic because its a fixed point of the Prym involution.
\end{proof}

\subsection{Genus 4}\label{subsec:g4}
Consider the X-shaped translation surface $X(w, e)$ as in \autoref{fig:prym4-disc53}.
\begin{figure}[t]
    \centering
    \includegraphics[width=\textwidth]{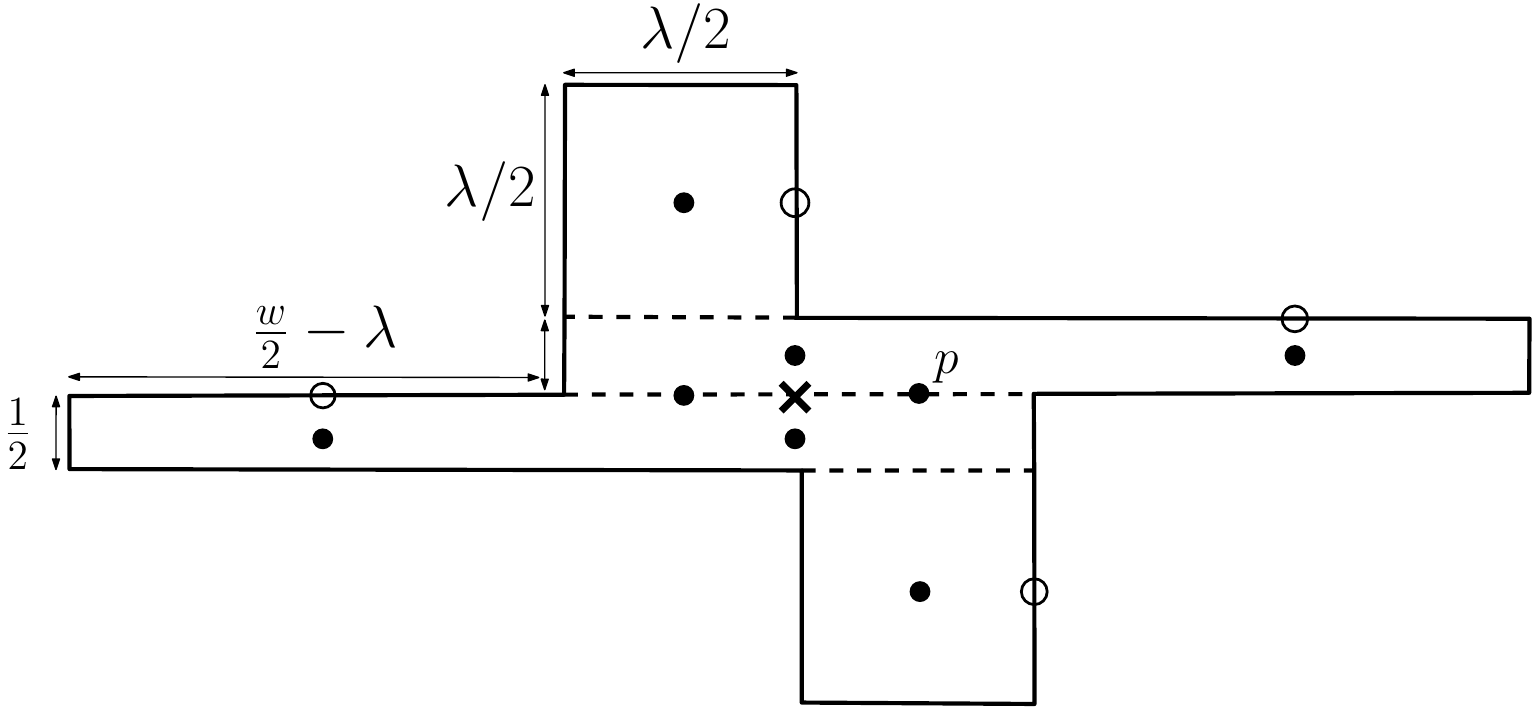}
    \caption{The $X$-shaped prototype $X(w, e)$.}
    \label{fig:prym4-disc53}
\end{figure}
For integers $w \in \NN$ and $e \in \ZZ$, it consists of two $\lambda/2 \times \lambda/2$ squares attached to two $w/2 \times 1/2$ rectangles, where $\lambda = (e + \sqrt{e^2 + 4w})/2$.
In each direction the surface decomposes into four horizontal cylinders, two \emph{short} simple cylinders and two \emph{long} cylinders of multiplicity one.
McMullen \cite{McMullenPrym} showed that particular choices of $w$ and $e$ yield Veech surfaces in the locus $\Omega E_D(6)$ that generate Teichm\"uller curves of discriminant $D = e^2 + 4w$.

In contrast with genus 2 and 3, Lanneau--Nguyen \cite{lanneau_nguyen_2020} showed that for nonsquare discriminants $D$ the loci $\Omega E_D(6)$ are nonempty and connected.
When $D > 5$, they showed that every eigenform can be $\SL(2, \RR)$-normalized to a prototype $X(w, e)$.
(In the remaining case of $D = 5$, they construct another explicit prototype in $\Omega_5(6)$.)
These prototypes $X(w, e)$ also admit butterfly moves $B_k$; the butterfly move $B_\infty$ is always admissible, and the move $B_2$ is admissible whenever $(e + 8)^2 < D$.

We now call a prototype $(w, e)$ \emph{good} if either
\begin{itemize}
    \item $w \equiv 1 \pmod{2}$ and $B_2$ is admissible, or
    \item $w \equiv 2 \pmod{4}$ and $B_4$ is admissible.
\end{itemize}
As in genus 3, we have good prototypical eigenforms $X(w, e)$ in each locus $\Omega E_D(6)$ for $D$ sufficiently large:
\begin{lemma}\label{lem:g4--good-prototypes}
    Let $S^4_{\text{bad}} \coloneqq \{ 5, 8, 12, 13, 17, 21, 24, 32, 33, 41, 57, 65, 73, 97, 113 \}$.
    For any $D \notin S^4_{\text{bad}}$, then there is a good prototype $(w, e)$ with discriminant $D$.
\end{lemma}
For a proof, see \autoref{appendix:g4--good-prototypes}.

\begin{proof}[Proof of \autoref{thm:main_thm} in genus 4]
If $D = 5$, we note that $\Omega E_5(6)$ is the $\SL(2, \RR)$-orbit of the Bouw-M\"oller surface $B(3, 5)$, which B. Wright \cite{wright2021periodic}*{Theorem 5.1} proved has periodic points the fixed points of the Prym involution.
If $D \in S^4_{\text{bad}}$, then we can use the algorithm of Chowdhury--Everett--Freedman--Lee \cite{icerm} to compute the periodic points of any explicit prototype $X(w, e)$ in $\Omega_D(6)$.
Otherwise, we compute the periodic points of a good prototypical eigenform $X(w, e)$ in the unique connected component of $\Omega_D(6)$.
Let $B_q$ denote either $B_2$ when $w \equiv 1 \pmod{2}$ or $B_4$ when $w \equiv 2 \pmod{4}$.

Applying \autoref{lem:two-regions-mult-one} to the two long horizontal and vertical cylinders, we find that the interior points of these cylinders that are potentially periodic are the points marked with solid dots in \autoref{fig:prym4-disc53}.
We claim that none of them points are periodic.
By using Dehn twists, the Prym involution, and the $B_\infty$ butterfly move, it suffices to show that the point $p$ in \autoref{fig:prym4-disc53} is not periodic:

\begin{lemma}\label{lem:g4-remove-interior-per-pt}
    The point $p$ in \autoref{fig:prym4-disc53} is not periodic.
\end{lemma}
\begin{proof}
    Consider the four-cylinder decomposition in the direction of slope $1/\lambda$ as shown in \autoref{fig:g4-1-lmbd.pdf}:
    \begin{figure}[ht]
    \centering
    \includegraphics[width=\textwidth]{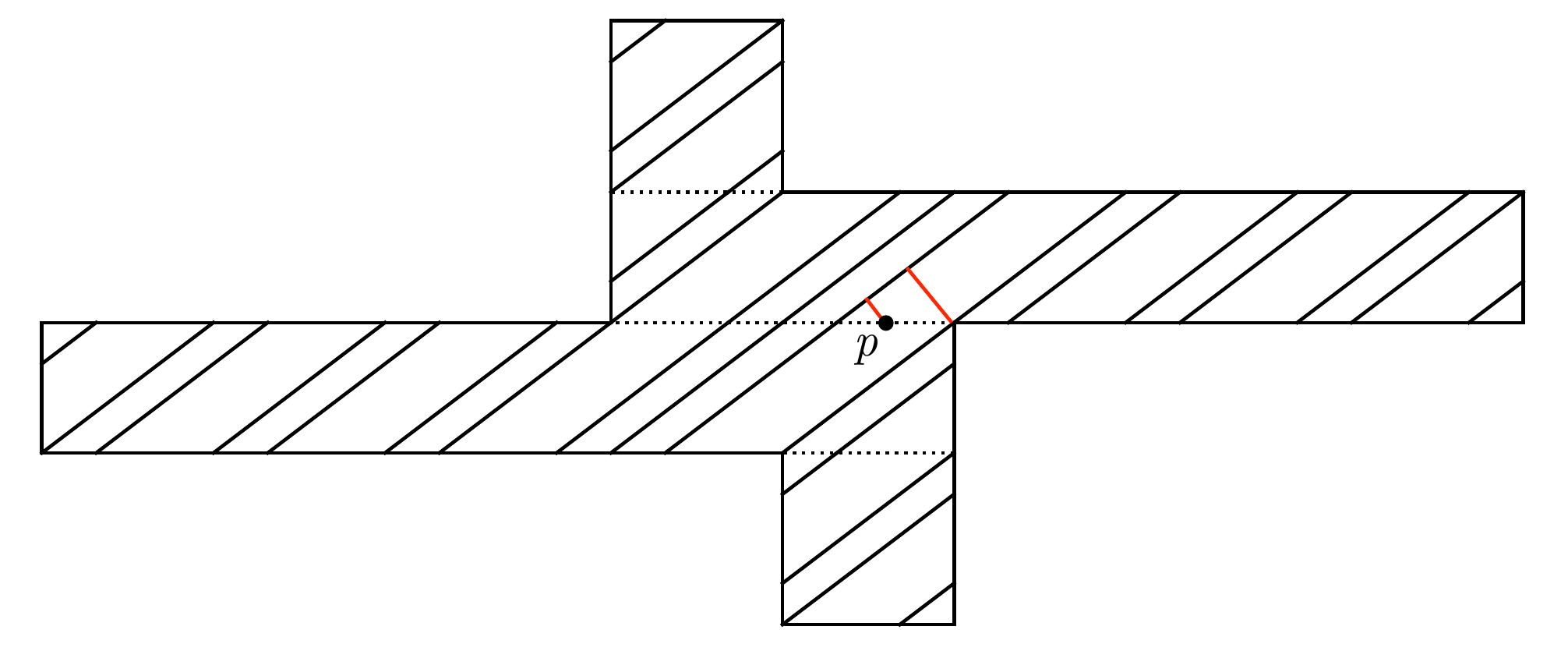}
    \caption{A cylinder decomposition in the direction $1/\lambda$ for the prototype $X(7, -4)$.}
    \label{fig:g4-1-lmbd.pdf}
    \end{figure}
    It consists of two small cylinders $C_S$ and two large cylinders $C_L$, none of which are simple.
    (See Lanneau--Nguyen \cite{lanneau_nguyen_2020}*{Lemma B.1}.)
    Let $w_S$ and $w_L$ denote the widths of the smaller and larger cylinders, respectively.
    Then we have $w_S + w_L = \lambda/2$, and $n w_S + (n - 1)w_L = w/2 - \lambda$ where $n \in \NN$ measures the winding of saddle connections in the long horizontal cylinders.
    This implies $w_L = (n+1)\lambda/2 - w/2$.
    The Rational Height Lemma and similar triangles imply
    \begin{equation*}
        \frac{h(C_L, p)}{h(C_L)} = \frac{\lambda/4}{w_L} = \frac{\lambda/4}{(n+1)\lambda/2 - w/2} \notin \mathbb{Q},
    \end{equation*}
    meaning $p$ is not a periodic point.
\end{proof}

In summary, none of the interior points of the four long horizontal and vertical cylinders are periodic.
It remains to classify periodic points on the long horizontal saddle connections of the long horizontal cylinders and the long vertical saddle connections of the long vertical cylinders.
As before, by symmetry it suffices to consider one horizontal boundary saddle connection.
Using the butterfly move $B_2$ (which is admissible in either case) and repeating the calculations as in the \autoref{subsec:g3}, it remains to rule out the midpoint $m$ of the horizontal saddle connection.
\begin{figure}[tb]
    \centering
    \includegraphics[width=\textwidth]{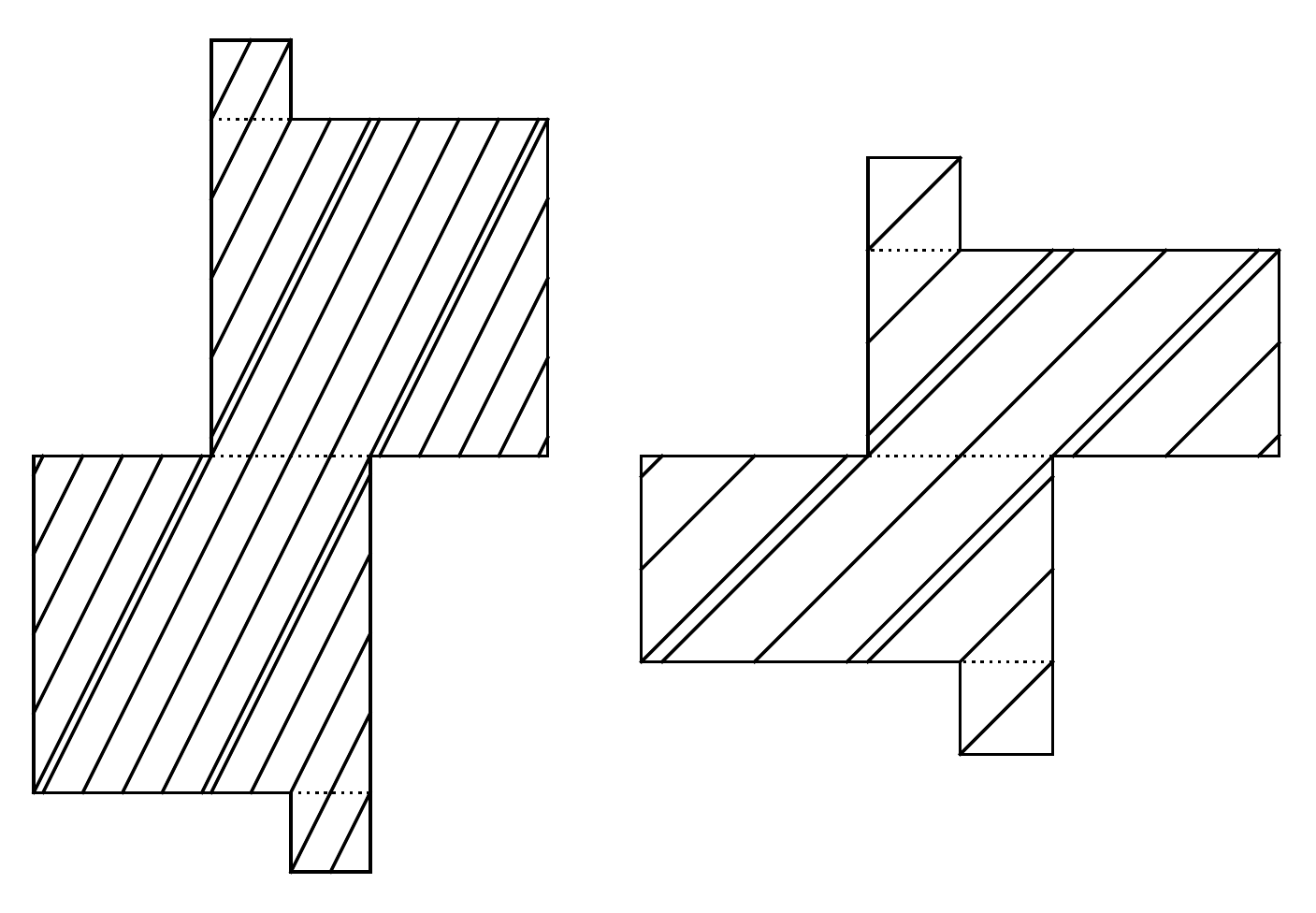}
    \caption{Cylinder decompositions from the $B_2$ butterfly move on the prototypes $X(1, -4)$ and $X(2, -4)$}
    \label{fig:g4-b2-decompositions}
\end{figure}
Consider the geodesic segment $\gamma$ starting at the interior Prym fixed point and ending at $m$ having slope $q / w$; see \autoref{fig:g4--fake-wp} for two examples.
\begin{figure}[tb]
    \centering
    \includegraphics[width=\textwidth]{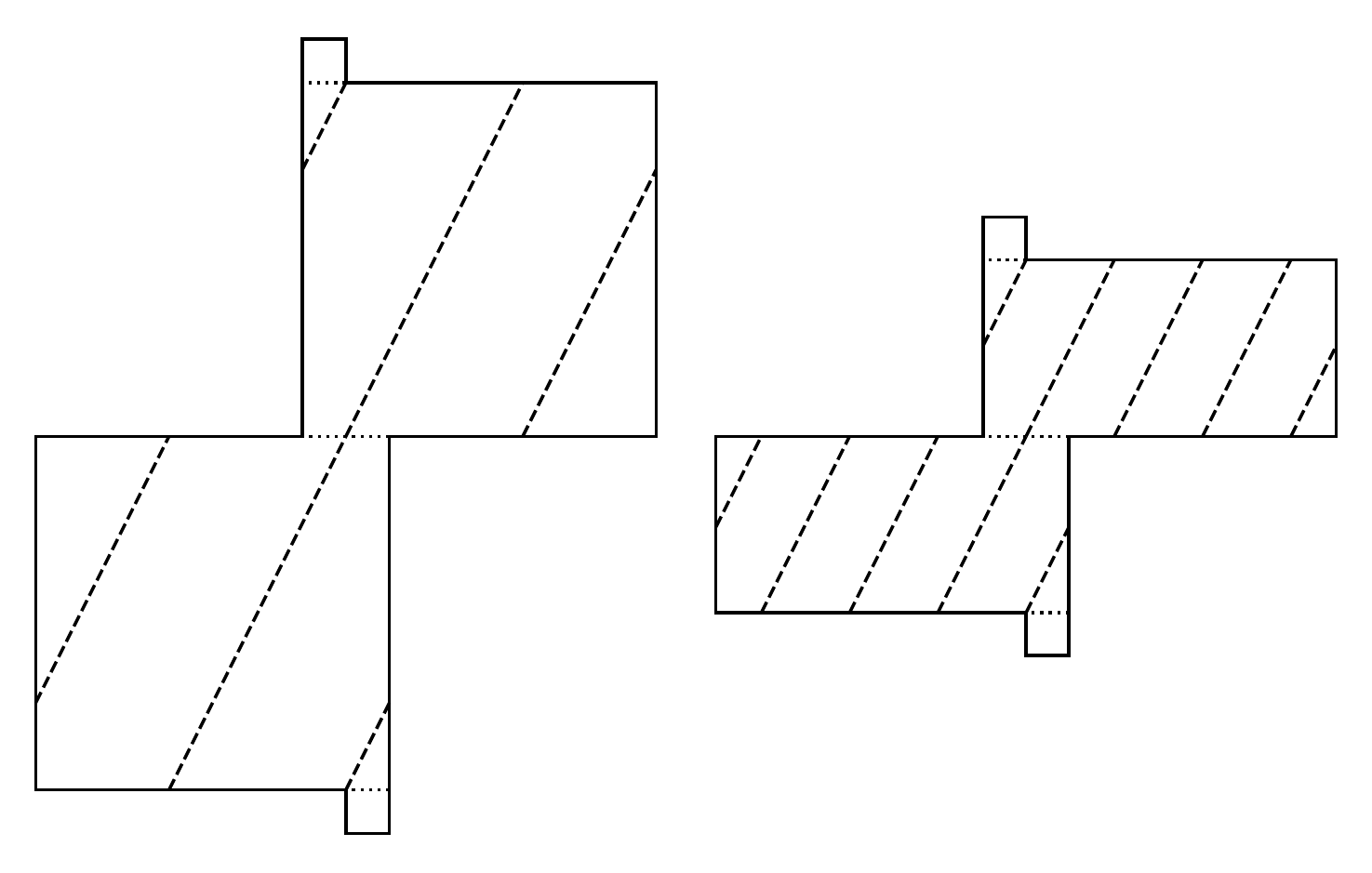}
    \caption{The geodesic connecting $m$ to $-m$ maps under the butterfly move $B_q$ to the horizontal saddle connection through the interior fixed point.}
    \label{fig:g4--fake-wp}
\end{figure}
Note that $\gamma$ has holonomy vector $[w/4, q/4]^T$.
Under the move $B_q$, $\gamma$ becomes the horizontal saddle connection through the interior fixed point.
Moreover, the midpoint $m$ must map halfway between the interior fixed point and the cone point on the new surface $B_q(w, e)$, i.e., to the point $p$ in \autoref{lem:g4-remove-interior-per-pt}.
We can then repeat the argument in \autoref{lem:g4-remove-interior-per-pt} on the prototype $B_q(w, e)$, noting that $B_q(w, e)$ has no tilt by \autoref{lem:g4--no-tilt}.
(We caution that $B_4(w, e) = (w', 2, 0, e')$, i.e., the height of the new long horizontal cylinder is $2$ and not $1$.
But the value of $h$ does not enter the calculations.)
We conclude that $m$ itself is not a periodic point.
\end{proof}

\appendix
\section{Good Prototypes}\label{appendix:good-prototypes}
\subsection{Genus 3}\label{appendix:g3--good-prototypes}
\begin{proof}[Proof of \autoref{lem:g3--good-prototypes}]
    We use explicit computer calculation to search for good prototypes in all discriminants $D \le 900$; the discriminants for which there are none are those in the statement.
    Now assume $D > 900$.
    Every admissible discriminant $D$ appears uniquely in the first column of \autoref{tab:g3--good-prototypes}.
        \begin{table}[ht]
        \centering
        \begin{tabular}{@{}cccc@{}}
        \toprule
        $D$     & $e$   & $w \coloneqq (D - e^2)/8$   & $q$ \\ \midrule
        $1 \pmod{16}$  & $\pm 3$ & $1 \pmod{2}$ & $2$ \\
        $8 \pmod{16}$  & $0$     & $1 \pmod{2}$ & $2$ \\
        $9 \pmod{16}$  & $\pm 1$ & $1 \pmod{2}$ & $2$ \\
        $12 \pmod{16}$ & $-2$    & $1 \pmod{2}$ & $2$ \\ \midrule
        $0 \pmod{32}$  & $-4$    & $2 \pmod{4}$ & $4$ \\
        $16 \pmod{32}$ & $0$     & $2 \pmod{4}$ & $4$ \\
        $20 \pmod{32}$ & $-2$    & $2 \pmod{4}$ & $4$ \\ \midrule
        $4 \pmod{64}$  & $-6$    & $4 \pmod{8}$ & $8$ \\
        $36 \pmod{64}$ & $-2$    & $4 \pmod{8}$ & $8$ \\ \bottomrule
        \end{tabular}%
        \caption{Good prototypes $(w, e)$ for each discriminant $D > 900$.}
        \label{tab:g3--good-prototypes}
        \end{table}
    For each $D$ choose $e$ as in the second column, and observe that $w \coloneqq (D - e^2)/8$ is an integer satisfying the congruence in the third column.
    The pair $(w, e)$ is a prototype so long as $e + 2 < w$, because the other arithmetic conditions in Lanneau-Nguyen \cite{LanneauNguyen}*{Theorem 4.1} hold automatically.
    But $e + 2 < 5$, and $w$ is at least $100$, so this inequality holds for $D$ sufficiently large.

    The butterfly move $B_q$ in column four is admissible so long as $(e + 4q)^2 < D$.
    The left-hand side of this inequality over all pairs $(e, q)$ occurring in \autoref{tab:g3--good-prototypes} is largest at $(-2, 8)$, and this holds when $D > 900$.
    This means that when $D > 900$, the prototype $(w, e)$ indicated in \autoref{tab:g3--good-prototypes} is good.
\end{proof}
\begin{lemma}\label{lem:g3--no-tilt}
    If $(w, e)$ is a good prototype, and $B_q$ is the corresponding good butterfly move, then $B_q(w, e)$ is a prototype without tilt.
\end{lemma}
\begin{proof}
    Since $w \equiv q/2 \pmod{q}$, we can write $w = q/2 + kq$ for some $k \in \ZZ$.
    We have $e_q \coloneqq -e - 4q$ and $h_q \coloneqq \gcd(q, w) = \gcd(q, q/2 + kq) = q/2$.
    Because $D = e^2 + 8w = e_q^2 + 8 w_q h_q$, we have
    \begin{align*}
        w_q \coloneqq \frac{D - e_q^2}{8h_q} = \frac{D - e^2 - 8eq - 16q^2}{4q} &= \frac{8w}{4q} - 2e  - 4q = 1 + 2k - 2e - 4q.
    \end{align*}
    Then we have
    \begin{align*}
        0 \le t_q < \gcd(w_q, h_q) = \gcd(1 + 2k - 2e - 4q, q/2) = \gcd(1 + 2k - 2e, q/2) = 1,
    \end{align*}
    since either $q / 2 = 1$, or $q / 2 \in \{2, 4\}$ yet $1 + 2k - 2q$ is odd.
    We conclude that $t_q = 0$ and $B_q(w, e)$ has no tilt.
\end{proof}

\subsection{Genus 4}\label{appendix:g4--good-prototypes}
\begin{proof}[Proof of \autoref{lem:g4--good-prototypes}]
    When $D \le 255$, we use computer search to show that every discriminant not in $S^4_\text{bad}$ has a good prototype.
    Now assume that $D > 255$.
    Each discriminant $D$ belongs to a unique row in \autoref{tab:g4--good-prototypes}; let $(w, e)$ be the indicated prototype.
    \begin{table}[ht]
        \centering
        \begin{tabular}{@{}lccc@{}}
        \toprule
        $D$           & $e$  & $w \coloneqq (D - e^2)/4$ & $q$ \\ \midrule
        $0 \pmod{8}$  & $-2$ & $1 \pmod{2}$              & $2$ \\
        $4 \pmod{8}$  & $0$  & $1 \pmod{2}$              & $2$ \\
        $5 \pmod{8}$  & $-1$ & $1 \pmod{2}$              & $2$ \\ \midrule
        $1 \pmod{16}$ & $-3$ & $2 \pmod{4}$              & $4$ \\
        $9 \pmod{16}$ & $-1$ & $2 \pmod{4}$              & $4$ \\ \bottomrule
        \end{tabular}
        \caption{Good prototypes in genus 4 for discriminant $D > 225$}
        \label{tab:g4--good-prototypes}
    \end{table}
    Then we see that $(w, e)$ satisfies all the arithmetic conditions in Lanneau--Nguyen \cite{lanneau_nguyen_2020}*{Proposition 2.2}.
    Indeed, we have $w > 0$, $D = e^2 + 4w$ by construction and $(e + 4)^2 < 16 \gg D$ so that $\lambda < w/2$.

    Observe that the largest value of $(e + 4q)^2$ over the pairs $(e, q)$ in the table is $(-3 + 4 \cdot 4)^2 = 225$.
    This means for all discriminants $D > 225$ we have the corresponding $B_q$ is admissible.
\end{proof}

\begin{lemma}\label{lem:g4--no-tilt}
    If $B_q$ with $q \in \{2, 4\}$ is the corresponding butterfly move for the good prototype $(w, e)$, then the prototype $B_q(w, e)$ has no tilt.
\end{lemma}
\begin{proof}
    Write $w = q/2 + kq$ for some $k \in \NN$.
    Via Lanneau--Nguyen \cite{lanneau_nguyen_2020}*{Proposition 2.7} we have $e_q = -e - 4q$ and $h_q = \gcd(q, w) = q/2$.
    Now
    \begin{align*}
        w_q = \frac{D - e_q^2}{4h_q} = \frac{D - e^2 - 8q - 16q^2}{2q} &= \frac{4w}{2q} - 4 - 8q = 1 + 2k - 4 - 8q.
    \end{align*}
    This implies that
    \begin{align*}
        0 \le t' < \gcd(w', h') = \gcd(1 + 2k - 4 - 8q, q/2) = \gcd(1 + 2k - 4, q/2).
    \end{align*}
    When $q = 2$ we have $\gcd(1 + 2k - 4, 1) = 1$, and when $q = 4$ we also have $\gcd(1 + 2k - 4, 2) = 1$.
    This implies that $t' = 0$ and $B_q(w, e)$ has no tilt as desired.
\end{proof}

\bibliography{references}
\bibliographystyle{abbrv}
\end{document}